\documentclass[a4paper, reqno, 11pt]{amsart}
\usepackage{color}
\makeatletter

\@addtoreset{equation}{section}
\makeatother
\usepackage{setspace}
\usepackage{xcolor}
\usepackage{here}
\usepackage{graphicx}
\usepackage{float}
\usepackage{hyperref}
\usepackage[T1]{fontenc}
\everymath{\displaystyle}
\usepackage{amsmath}
\usepackage{amssymb}
\usepackage{latexsym}
\usepackage{amsthm}
\usepackage{hyperref}
\usepackage{mathtools}
\usepackage{mathrsfs}
\newtheorem{Thm}{Theorem}[section]

\newtheorem{Lem}[Thm]{Lemma}
\newtheorem{Prop}[Thm]{Proposition}

\theoremstyle{definition}

\newtheorem{Rem}[Thm]{Remark}

\begin{document}

\begin{abstract}
We disprove a conjecture stated in a recent paper by Arnold and Villasenor concerning the sum and the maximum of independent and identically distributed half-normal random variables. 
Our method is applicable to generalized gamma distributions. 
\end{abstract}

\title[]{A remark on the comparison of the sum and the maximum of positive random variables}
\author{Kazuki Okamura}
\date{\today}
\address{Department of Mathematics, Faculty of Science, Shizuoka University, 836, Ohya, Suruga-ku, Shizuoka, 422-8529, JAPAN.}
\email{okamura.kazuki@shizuoka.ac.jp}
\maketitle

\section{Introduction}

Recently, Arnold and Villasenor \cite{AV2026} provided two proofs that the distributions of the sum and the maximum of two independent half-normal random variables are of the same type. 
More specifically, 
it holds that 
\begin{equation}\label{eq:AV}
X_1 + X_2 = \sqrt{2} \max\{X_1, X_2\} \quad \textup{ in distribution}
\end{equation}
if $X_1$ and $X_2$ are independent and have  a common half-normal distribution. 
They also showed that under some regularity conditions, \eqref{eq:AV} implies that $X_1$ has a half-normal distribution. 
In \cite[Section 4]{AV2026}, they conjectured that for $n \ge 3$, 
\begin{equation}\label{eq:AV-cinj}
X_1 + \cdots + X_n = (n!)^{1/n} \max\{X_1, \dots, X_n\} \quad \textup{ in distribution}
\end{equation}
if $X_1, \dots, X_n$ are independent and have  a common half-normal distribution. 
Both \eqref{eq:AV} and \eqref{eq:AV-cinj} are invariant under a common positive rescaling of the random variables. 
Thus, the scale parameter of the half-normal distribution is left unrestricted here.

The purpose of this note is to disprove the conjecture. 
Our arguments are applicable to a class of distributions containing all generalized gamma distributions, which are originally introduced by \cite{Stacy1962}. 

Assume that $X_i, i \ge 1$, are independent and identically distributed random variables supported on $(0,\infty)$. 
Let $S_n \coloneqq \sum_{i=1}^{n} X_i$ and $M_n \coloneqq \max_{1 \le i \le n} X_i$ for $n \ge 1$. 

Let $\Gamma$ be the gamma function and define 
$h(x) \coloneqq \frac{\log \Gamma(x+1)}{x}, x > 0$. 

\begin{Thm}\label{thm:dis-AV}
Assume that $X_1$ has a strictly positive density $f$ supported on
$(0,\infty)$ and that there exist four positive constants $c_1, c_2, \alpha$, and $\beta$ such that 
\begin{equation}\label{eq:condition-zero-density} 
\lim_{x \to +0} \frac{f(x)}{x^{\alpha-1}} = c_1.  
\end{equation}  
\begin{equation}\label{eq:condition-tail-density}
\lim_{x \to \infty} \frac{-\log f(x)}{x^{\beta}} = c_2.  
\end{equation}  
Assume that $n \ge 2$ and 
\begin{equation}\label{eq:beta-n-relation}
\beta < \frac{\log n}{\log n + h(\alpha) - h(n\alpha)}. 
\end{equation}
Then there is no constant $C$ such that $S_n = CM_n$ in distribution. 
\end{Thm}

If $X_1$ has the standard half-normal distribution, 
then \eqref{eq:condition-zero-density} and \eqref{eq:condition-tail-density} hold with $c_1 = \sqrt{2/\pi}, c_2 = 1/2, \alpha = 1$, and $\beta = 2$. 
For $\alpha = 1$ and $\beta = 2$, 
condition \eqref{eq:beta-n-relation} is equivalent to $n \ge 3$; 
this can be verified by induction on $n$. 
For every $\beta \in (0,1]$, \eqref{eq:beta-n-relation} holds for every $n \ge 2$. 
We remark that $\log n + h(\alpha) - h(n\alpha) > 0$ for every $n \ge 2$ by Proposition \ref{prop:h-increase} (i) below. 

We show Theorem \ref{thm:dis-AV} by contradiction. 
The main idea of the proof is to first compare the asymptotic behavior of $P(S_n \le x)$ and $P(M_n \le x)$ as $x \to +0$ and then compare the asymptotic behavior of $P(S_n > x)$ and $P(M_n > x)$ as $x \to +\infty$. 

\section{Proof}

First, we compare the asymptotic behavior of $P(S_n \le x)$ and $P(M_n \le x)$ as $x \to +0$. 
Let $C_{\alpha, n} \coloneqq \frac{\Gamma(n\alpha+1)^{1/(n\alpha)}}{\Gamma(\alpha+1)^{1/\alpha}}$. 

\begin{Lem}\label{lem:const-unique-zero}
If there exists a constant $C$ such that $S_n = CM_n$ in distribution, 
then $C = C_{\alpha,n}$. 
\end{Lem}

\begin{proof}
We see that for every $x > 0$, 
\[ P(S_n \le x) = \int_{D_1} \prod_{i=1}^{n} f(x_i) dx_1 \cdots dx_{n},  \]
where $D_1 \coloneqq \{(x_1, \dots, x_n) \in (0,\infty)^n \colon x_1 + \cdots + x_n < x \}$. 
By the change of variables $x_i = x u_i$, 
we obtain that for every $x > 0$, 
\[ P(S_n \le x) = x^{\alpha n} \int_{D_2} \prod_{i=1}^{n} \frac{f(x u_i)}{(x u_i)^{\alpha -1}} \prod_{i=1}^{n} u_i^{\alpha -1} du_1 \cdots du_{n},  \]
where $D_2 \coloneqq \{(u_1, \dots, u_n) \in (0,\infty)^n \colon u_1 + \cdots + u_n < 1 \}$. 
By \eqref{eq:condition-zero-density}, 
we can apply the dominated convergence theorem and obtain that 
\[ \lim_{x \to +0} \int_{D_2} \left| \prod_{i=1}^{n} \frac{f(x u_i)}{(x u_i)^{\alpha -1}}  - c_1^n \right|  \prod_{i=1}^{n} u_i^{\alpha -1} du_1 \cdots du_{n} = 0. \]
Using this together with the identity $\int_{D_2}  \prod_{i=1}^{n} u_i^{\alpha -1} du_1 \cdots du_{n} = \frac{\Gamma(\alpha)^n}{\Gamma(n \alpha +1)}$, 
we obtain that 
\begin{equation}\label{eq:sum-zero}
\lim_{x \to +0} \frac{P(S_n \le x)}{x^{\alpha n}}  = c_1^n \frac{\Gamma(\alpha)^n}{\Gamma(n \alpha +1)}. 
\end{equation} 

We see that for every $x > 0$, 
\[ P(M_n \le x) = P(X_1 \le x)^n =  x^{\alpha n} \left(\int_0^1 \frac{f(xu)}{(xu)^{\alpha-1}} u^{\alpha-1} du \right)^n.  \]
By \eqref{eq:condition-zero-density}, 
we can apply the dominated convergence theorem and obtain that 
\begin{equation}\label{eq:max-zero} 
\lim_{x \to +0} \frac{P(M_n \le x)}{x^{\alpha n}}  = \frac{c_1^n}{\alpha^n}. 
\end{equation}

By \eqref{eq:sum-zero} and \eqref{eq:max-zero}, we see that $C = C_{\alpha,n}$. 
\end{proof}

Second, we compare the asymptotic behavior of $P(S_n > x)$ and $P(M_n > x)$ as $x \to +\infty$. 
By Lemma \ref{lem:const-unique-zero}, 
$S_n = C_{\alpha, n}  M_n$ in distribution. 

Without loss of generality, we may assume that $c_2=1$. 
Indeed, for $\sigma>0$, let $Y_i = \sigma X_i$. 
The density of $Y_i$ is $\widetilde f(x) = \sigma^{-1}f(\sigma^{-1}x)$, and by \eqref{eq:condition-tail-density}, 
 $\lim_{x\to\infty}\frac{-\log \widetilde f(x)}{x^\beta} = \frac{c_2}{\sigma^\beta}$.
Taking $\sigma=c_2^{1/\beta}$ makes this limit equal to $1$. 
Moreover, the identity $S_n = C_{\alpha,n}M_n$ is preserved under this common rescaling. 
Therefore, we relabel the $Y_i$ as the $X_i$.

By the assumption \eqref{eq:condition-tail-density}, 
for every $\epsilon \in (0,1/2)$, it holds that for large $x$, 
$ \exp\left(- (1+\epsilon)^{\beta} x^{\beta} \right) \le f(x) \le \exp\left(- (1-\epsilon)^{\beta} x^{\beta} \right)$.

By integration by parts, 
it holds that for every $\beta > 0$ and $z > 0$, 
\begin{equation*}
\int_{z}^{\infty} \exp(-t^{\beta}) dt = \frac{\exp(-z^{\beta})}{\beta} z^{1 - \beta} - \frac{\beta - 1}{\beta} \int_{z}^{\infty} t^{-\beta} \exp(-t^{\beta})  dt. 
\end{equation*}
Hence, for large $z$, 
\[ \int_{z}^{\infty} \exp(-t^{\beta}) dt \le \frac{\exp(-z^{\beta})}{\beta} z^{1 - \beta} + \frac{|\beta - 1|}{\beta} z^{-\beta}  \int_{z}^{\infty} \exp(-t^{\beta})  dt. \]
We obtain that 
\[  \int_{z}^{\infty} \exp(-t^{\beta}) dt = O\left(  z^{1 - \beta} \exp(-z^{\beta}) \right), \quad z \to \infty, \]
and furthermore 
\[ \int_{z}^{\infty} t^{-\beta} \exp(-t^{\beta}) dt = o\left( z^{1 - \beta} \exp(-z^{\beta}) \right), \quad z \to \infty. \] 

Hence we see that for large $x$, 
\begin{equation}\label{eq:tail-X-lower}
P(X_1 > x)  \ge \frac{1}{2} \frac{\exp(- (1+\epsilon)^{\beta} x^{\beta})}{\beta (1+\epsilon)^{\beta} x^{\beta-1}}
\end{equation}
and
\begin{equation}\label{eq:tail-X-upper}
P(X_1 > x)  \le 2\frac{\exp(- (1-\epsilon)^{\beta} x^{\beta})}{\beta (1-\epsilon)^{\beta} x^{\beta-1}}.
\end{equation}

We give a lower bound for $P(M_n > x)$. 
Since $f$ is strictly positive on $(0,\infty)$, 
$P(X_1 > x) > 0$ for every $x > 0$. 
We see that 
\begin{equation}\label{eq:ratio-max-single} 
\lim_{x \to \infty} \frac{P(M_n > x)}{P(X_1 > x)} = n. 
\end{equation}

By \eqref{eq:tail-X-lower} and \eqref{eq:ratio-max-single}, 
we see that for large $x$, 
\begin{equation}\label{eq:max-infinity-lower}
P(M_n > x) \ge \frac{n \exp(- (1+\epsilon)^{\beta} x^{\beta})}{4\beta (1+\epsilon)^{\beta} x^{\beta-1}}.
\end{equation}

We give an upper bound for $P(S_n > x)$.

We first consider the case where $\beta \ge 1$. 
Since $\beta \ge 1$, by convexity, 
we see that $\sum_{i=1}^{n} X_i^{\beta} \ge n^{1-\beta} S_n^{\beta}$. 
By this and the exponential Chebyshev inequality, 
we see that for every $\epsilon \in (0,1)$, $x > 0$ and $n \ge 1$, 
\begin{align}\label{eq:sum-infinity-upper}
P(S_n > x) &\le P\left( \sum_{i=1}^{n} X_i^{\beta} \ge n^{1-\beta} x^{\beta} \right) \notag\\
&\le \exp\left(-(1-\epsilon)^{\beta} n^{1-\beta} x^{\beta} \right) E\left[ \exp\left((1-\epsilon)^{\beta}X_1^{\beta} \right) \right]^n.
\end{align} 
By \eqref{eq:condition-tail-density}, 
$E\left[ \exp\left((1-\epsilon)^{\beta} X_1^{\beta} \right) \right]< \infty$. 
By \eqref{eq:beta-n-relation}, it holds that $C_{\alpha,n}^{-\beta} < n^{1-\beta}$. 
Now we take a sufficiently small $\epsilon > 0$ such that $(1+\epsilon)^{\beta} C_{\alpha,n}^{-\beta} < (1-\epsilon)^{\beta} n^{1-\beta}$. 
By this, \eqref{eq:max-infinity-lower} and \eqref{eq:sum-infinity-upper}, 
we obtain that 
\[ \lim_{x \to \infty} \frac{P\left(M_n > x/C_{\alpha,n} \right)}{P(S_n > x)} = \infty. \]
This contradicts $S_n = C_{\alpha,n} M_n$ in distribution. 
The proof of Theorem \ref{thm:dis-AV} in the case where $\beta \ge 1$ is completed.

We next consider the case where $0 < \beta < 1$. 
By \eqref{eq:condition-tail-density}, $E[\exp(t X_1)] = \infty$ for every $t > 0$. 
By \cite[Theorem 2]{DFK2008}, 
\begin{equation}\label{eq:subexp} 
\liminf_{x \to \infty} \frac{P(S_n > x)}{P(M_n > x)} < \infty. 
\end{equation}

By Lemma \ref{lem:const-unique-zero}, if there exists a constant $C$ such that $S_n = CM_n$ in distribution, 
then $C = C_{\alpha,n}$. 
By \eqref{eq:subexp} and \eqref{eq:ratio-max-single}, 
we obtain that 
\begin{equation}\label{eq:ratio-contradict}
\limsup_{x \to \infty} \frac{P(X_1 > x)}{P(X_1 > x/C_{\alpha,n})} =  \limsup_{x \to \infty} \frac{P(M_n > x)}{P(M_n > x/C_{\alpha,n})} > 0. 
\end{equation}

By Lemma \ref{lem:property-digamma} (i) below, $h(x)$ is strictly increasing.  
Hence $C_{\alpha,n} > 1$ for $n \ge 2$. 
By \eqref{eq:tail-X-lower} and \eqref{eq:tail-X-upper},  using a sufficiently small $\epsilon > 0$, 
\[ \lim_{x \to \infty} \frac{P(X_1 > x)}{P(X_1 > x/C_{\alpha,n})} = 0. \]
This contradicts \eqref{eq:ratio-contradict}. 
The proof of Theorem \ref{thm:dis-AV} in the case where $0 < \beta < 1$ is completed. 

The following assertions provide properties of $h$. 

\begin{Lem}\label{lem:property-digamma}
(i) $h$ is strictly increasing on $(0,\infty)$.\\
(ii) $g(x) \coloneqq h(\exp(x))$ is strictly convex on $\mathbb R$. 
\end{Lem}

\begin{proof}
(i) Let $\psi$ be the digamma function. 
Then $\frac{d}{dx} \log \Gamma(x+1) = \psi(x+1)$. 

By the Bohr--Mollerup theorem, 
$\log \Gamma(x+1)$ is strictly convex. 
Hence, the secant slope $\frac{\log\Gamma(x+1)-\log\Gamma(1)}{x} = h(x)$ is strictly increasing on $(0,\infty)$.  

(ii) 
Since $h(t) = \frac{1}{t} \int_0^{t} \psi(s+1) ds$, 
we see that 
\[ t h^{\prime} (t) = \psi(t+1) - \frac{1}{t} \int_0^{t} \psi(s+1) ds = \frac{1}{t} \int_0^t s \psi^{\prime}(s+1) ds.  \]
Since $g^{\prime}(x) = \exp(x) h^{\prime}(\exp(x))$, 
it suffices to show that $q(s) \coloneqq s  \psi^{\prime}(s+1)$ is strictly increasing. 
By an integral expression of the trigamma function (\cite[Formula 6.4.1, p.260]{AS1964}), 
\[ \psi^{\prime}(s+1) = \int_0^{\infty}  \frac{r \exp(-sr)}{\exp(r) -1} dr. \]
Hence
\begin{equation}\label{eq:integral-q} 
q(s) = \int_0^{\infty}  \exp(-r) \rho\left(\frac{r}{s} \right) dr,  
\end{equation}
where $\rho(t) \coloneqq \frac{t}{\exp(t) -1}$. 
Since $\rho$ is strictly decreasing, $q$ is strictly increasing. 
\end{proof}

\begin{Prop}\label{prop:h-increase}
Let $a_n \coloneqq \frac{h(n \alpha) -h(\alpha)}{\log n}$ for $n \ge 2$. 
Then \\
(i) $a_n < 1$ for every $n$ and $\lim_{n \to \infty} a_n = 1$.\\
(ii) $(a_n)_n$ is strictly increasing. \\
(iii) There exists $N(\alpha,\beta) \ge 2$ such that \eqref{eq:beta-n-relation} holds if and only if  $n \ge N(\alpha,\beta)$. 
\end{Prop}

\begin{proof}
(i) By \eqref{eq:integral-q}, $q(s) < 1$ for every $s > 0$. 
Hence 
\[ g^{\prime}(x) = \frac{1}{\exp(x)} \int_0^{\exp(x)} q(s) ds < 1, \quad x \in \mathbb{R}. \] 
Since $a_n = \frac{g(\log (n\alpha)) - g(\log \alpha)}{\log(n \alpha) - \log \alpha}$, $a_n < 1$. 
By Stirling's formula, 
\[ h(x) = \frac{\log \Gamma(x+1)}{x} = \log x - 1 + O\left(\frac{\log x}{x} \right), \quad x \to \infty. \]
Hence $\lim_{n \to \infty} a_n = 1$. 

(ii) Let $b_n \coloneqq h((n+1) \alpha) -h(n\alpha)$ and $c_n \coloneqq \log (n+1) - \log n$ for $n \ge 1$. 
Since $a_n = \frac{\sum_{i=1}^{n-1} b_i}{\sum_{i=1}^{n-1} c_i}$, it suffices to show that $\left(\frac{b_n}{c_n} \right)_n$ is strictly increasing. 
We see that 
\[ \frac{b_n}{c_n} = \frac{g(\log((n+1)\alpha)) - g(\log(n\alpha))}{\log((n+1)\alpha) - \log(n\alpha)}. \]

By the increasing secant slope property for strictly convex functions, 
$\left(\frac{b_n}{c_n} \right)_n$ is strictly increasing.

(iii) This follows from (i) and (ii). 
\end{proof}

\begin{table}[H]
\centering
\begin{tabular}{|c|ccccc|}
\hline
 $\alpha \backslash \beta$ &  1.0 & 1.5 & 2.0 & 2.5 & 3.0  \\
\hline
0.25 &  2 & 7 &  34 &  117 & 356  \\
0.5 &    2 & 2 &  9 &  24  & 60 \\
0.75 &  2 & 2 &  4 &  10 & 23  \\
1.0 &    2 &  2 &  3 &  6 & 12  \\
1.25 &  2 & 2 &  2 &  4 & 7  \\
1.5 &    2 &  2 &  2 &  3  & 5\\
\hline  
\end{tabular}
\caption{Values of $N(\alpha,\beta)$.}\label{tab:CI3}
\end{table}

\begin{Rem}
Let $n=3$. 
For the standard half-normal distribution, we have an alternative proof. 
It is easy to see that 
\[ E\left[S_3^2 \right] = E[(X_1 + X_2 + X_3)^2] = 3 E[X_1^2] + 6 \left(E[X_1] \right)^2 = 3 + \frac{12}{\pi}. \] 

We also need the following lemma. 
\begin{Lem}
\[ E\left[M_3^2 \right] = 1 + \frac{2\sqrt{3}}{\pi}.  \]
\end{Lem}

\begin{proof}
Let $f$ and $F$ be the density and distribution functions of $X_1$ respectively. 
Then the density function of $M_3$ is given by $g(x) = 3 f(x) F(x)^2$. 
Since $x f(x) = -f^{\prime}(x)$, it holds that 
\[ E\left[M_3^2 \right] = 3 \int_0^{\infty} x^2 f(x) F(x)^2 dx  = -3 \int_0^{\infty} x F(x)^2 f^{\prime}(x) dx. \]
By integration by parts, 
\[ \int_0^{\infty} x F(x)^2 f^{\prime}(x) dx = - \int_0^{\infty} (F(x)^2 + 2x f(x) F(x)) f(x) dx.  \]
Since $2x f(x)^2 = - \frac{d}{dx} f(x)^2$, we obtain that 
\[ \int_0^{\infty} (F(x)^2 + 2x f(x) F(x)) f(x) dx = \frac{1}{3} - \int_0^{\infty} F(x) (f(x)^2)^{\prime} dx = \frac{1}{3} + \int_0^{\infty} f(x)^3 dx.  \]
Hence, 
\[ E\left[M_3^2 \right] = 1 + 3 \int_0^{\infty} f(x)^3 dx = 1 + \frac{2\sqrt{3}}{\pi}. \]
\end{proof}

By Lemma \ref{lem:const-unique-zero}, 
if there exists a constant $C$ such that $S_3 = CM_3$ in distribution, 
then $C$ would be equal to $6^{1/3}$ and hence $3 + \frac{12}{\pi} = 6^{\frac{2}{3}} \left(1 + \frac{2\sqrt{3}}{\pi} \right)$. 
Then $\pi \in \mathbb{Q}\left(\sqrt{3}, 6^{1/3}\right)$, which contradicts the fact that $\pi$ is a transcendental  number. 
We see that $\frac{3 + \frac{12}{\pi}}{1 + \frac{2\sqrt{3}}{\pi}}$ is approximately $3.24$ and $6^{\frac{2}{3}}$ is approximately $3.30$; they are numerically close to each other. \\
\end{Rem}

\noindent{\it Acknowledgments.} \ The author would like to express his gratitude to two anonymous referees for their helpful comments.  
The author is supported by JSPS KAKENHI JP22K13928. \\

\bibliographystyle{plain}
\bibliography{disprove-AV}

\end{document}